\documentclass[10pt,a4paper, xcolor]{article}
\usepackage{amsfonts}
\usepackage{amssymb,amsthm, amsmath,graphicx,bm,amsthm,mathtools }
\usepackage{mathrsfs}
\usepackage{color,enumerate,lineno}
\usepackage{url,hyperref}
\usepackage{color,enumerate}
\usepackage{url,lineno}
\usepackage[latin1]{inputenc}
\usepackage[all]{xy}

\newtheorem{theorem}{Theorem}[section]

\newtheorem{proposition}[theorem]{Proposition}
\newtheorem{corollary}[theorem]{Corollary}
\newtheorem{lemma}[theorem]{Lemma}
\newtheorem{algorithm}[theorem]{Algorithm}

\theoremstyle{remark}
\newtheorem{note}[theorem]{Note}
\newtheorem{example}[theorem]{Example}

\renewcommand{\L}{\mathcal{L}}

\def\sA{\mathscr{A}}
\def\sB{{\mathscr{B}}}

\def\scrP{{\mathscr{P}}}

\def\CC{\mathscr{C}}

\def\e{{\mathrm e}}
\def\g{{\mathrm g}}
\def\d{{\mathrm d}}
\def\h{{\mathrm h}}
\def\l{{\mathrm{l}}}
\def\m{{\mathrm m}}
\def\n{{\mathrm n}}
\def\t{{\mathrm t}}

\def\C{{\mathrm C}}
\def\F{{\mathrm F}}
\def\G{{\mathrm G}}

\def\rP{{\mathrm P}}

\def\Cad{{\mathrm {Cad}}}

\def\PF{{\mathrm {PF}}}

\def\rN{{\mathrm N}}

\def\Ap{{\mathrm{ Ap}}}
\def\max{{\mathrm{ max}}}

\def\gcd{{\mathrm{gcd}}}

\def\max{{\mathrm{max}}}
\def\msg{{\mathrm{ msg }}}

\def\Max{{\mathrm{Max}}}

\def\CL{\mathscr{L}}

\def\CL{\mathscr{L}}
\def\sR{\mathscr{R}}

\def\CC{\mathscr{C}}

\def\N{\mathbb{N}}

\def\Z{\mathbb{Z}}

\def\d{\mathrm{d}}
\def\ii{\mathrm{i}}
\def\r{\mathrm{r}}
\def\l{\mathrm{l}}

\def\L{\mathrm{L}}

\def\rank{\mathrm{rank}\, }
\def\Ap{\mathrm{Ap}}
\def\SG{\mathrm{SG}}
\def\Parf{\mathrm{Parf}}
\def\Psat{\mathrm{Psat}}
\def\Arf{\mathrm{Arf}}
\def\Sat{\mathrm{Sat}}

\def\int{\mathrm{int}}

\title{The covariety of perfect numerical semigroups with  fixed Frobenius number}

\author{
	M. A. Moreno-Fr\'{\i}as \footnote{
		Dpto. de Matem\'aticas, Facultad de Ciencias,
		Universidad de C\'adiz, E-11510, Puerto Real  (C\'{a}diz, Spain).
		Partially supported by  Junta de Andaluc\'{\i}a group FQM-298, 
		Proyecto de Excelencia de la Junta de Andalucía ProyExcel\_00868, Proyecto de investigación del Plan Propio--UCA 2022-2023 (PR2022-011) and Proyecto de investigación del Plan Propio--UCA 2022-2023 (PR2022-004).
		E-mail: mariangeles.moreno@uca.es.}
	\and
	J. C. Rosales \footnote{
		Dpto. de \'Algebra, Facultad de Ciencias, Universidad de Granada,
		E-18071, Granada. (Spain).
		Partially supported by  Junta de Andaluc\'{\i}a group FQM-343,
		Proyecto de Excelencia de la Junta de Andalucía ProyExcel\_00868 and Proyecto de investigación del Plan Propio--UCA 2022-2023 (PR2022-011).
		E-mail: jrosales@ugr.es.}
}

\date{}

\begin{document}

\maketitle

\begin{abstract}
	
	Let $S$ be a numerical semigroup.  We will say that  $h\in \N \backslash S$ is an {\it isolated gap }of $S$   if
	$\{h-1,h+1\}\subseteq S.$ A numerical semigroup  without isolated gaps
	is called perfect numerical semigroup.
	Denote by $\m(S)$ the multiplicity of a numerical semigroup $S$.  A covariety is a nonempty family $\CC$ of numerical semigroups that  fulfills  the following conditions: there is the minimum of $\CC,$ the intersection of two elements of  $\CC$  is again an element of  $\CC$ and $S\backslash \{\m(S)\}\in  \CC$  for all $S\in  \CC$ such that  $S\neq \min(\CC).$ 
	
	In this work we prove that the set $\scrP(F)=\{S\mid S \mbox{ is a perfect numerical}\\ \mbox{
		semigroup with Frobenius number }F\}$
	is a covariety. Also, we describe three algorithms which compute: the set $\scrP(F),$ the maximal elements of $\scrP(F)$ and the elements of $\scrP(F)$ with a given genus. 
	
	 A $\Parf$-semigroup (respectively, $\Psat$-semigroup) is a perfect numerical semigroup that in addition is an Arf numerical semigroup (respectively, saturated numerical semigroup). 
	We will prove that the sets:
 $\Parf(F)=\{S\mid S \mbox{ is a $\Parf$-numerical semigroup with Frobenius number} F\}$\,\,\, and \\ $\Psat(F)=\{S\mid S \mbox{ is a  $\Psat$-numerical semigroup with Frobenius number }\\
 F\}$	
are covarieties. As a consequence we present some algorithms to compute $\Parf(F)$ and $\Psat(F).$

\smallskip
    {\small \emph{Keywords:} Perfect numerical semigroup, saturated numerical semigroup, Arf numerical semigroup, 
     covariety, Frobenius number, genus,  algorithm. }

    \smallskip
    {\small \emph{MSC-class:} 20M14, 11D07, 13H10. }
\end{abstract}

\section{Introduction}

\hspace{0.42cm}Let $\Z$ be the set of integers and $\N=\{z\in \Z \mid z\ge 0\}$. A {\it submonoid} of $(\N,+)$ is a subset  of $\N$ which is closed under addition and contains the element $0.$ A {\it numerical semigroup} is a
submonoid $S$ of $(\N,+)$ such that $\N\backslash
S=\{x\in \N \mid x \notin S\}$ has finitely many elements. 

If $S$ is a numerical semigroup, then $\m(S)=\min(S\backslash \{0\})$, $\F(S)=\max\{z\in \Z \mid z \notin S\}$ and $\g(S)=\sharp(\N \backslash S)$ (where $\sharp X $ denotes the cardinality of a set $X$) are 
three important invariants of $S$, called   the {\it multiplicity}, the {\it Frobenius number} and the {\it genus} of $S$, respectively.

If $A$ is a nonempty subset  of $\N$, we denote by $\langle A
\rangle$ the submonoid of $(\N,+)$ generated by $A$. That is,
$\langle A \rangle=\{\lambda_1a_1+\dots+\lambda_na_n \mid n\in
\N\backslash \{0\}, \, \{a_1,\dots, a_n\}\subseteq A \mbox{ and }
\{\lambda_1,\dots,\lambda_n\}\subseteq \N\}.$  In \cite[Lema 2.1]{libro}) it is shown that $ \langle A \rangle$
is a numerical semigroup if and only if $\gcd(A)=1.$ 

If $M$ is a submonoid of $(\N,+)$ and $M=\langle A \rangle$, then we
say that $A$ is a {\it system of generators} of $M$. Moreover, if $M\neq
\langle B \rangle$ for all $B \varsubsetneq A$, then we will say
that $A$ is a {\it minimal system of generators} of $M$. In
\cite[Corollary 2.8]{libro} is shown that every submonoid of
$(\N,+)$ has a unique minimal system of generators, which in
addition it is finite. We denote by $\msg(M)$ the minimal system of
generators of $M$. The cardinality of $\msg(M)$ is called the {\it
	embedding dimension }of $M$ and will be denoted by $\e(M).$

The Frobenius problem (see \cite{alfonsin})  focuses on finding formulas to calculate the Frobenius number and the genus of a numerical semigroup from its minimal  system of generators. The problem was solved in \cite{sylvester} for numerical semigroups with embedding dimension two.  Nowadays, the problem is still open in the case of numerical semigroups with embedding dimension  greater than or equal to three. Furthemore, in this case the problem of computing the Frobenius number of a general numerical semigroup becomes NP-hard (see \cite{alfonsin2}).

Let $S$ be a numerical semigroup.  We will say that  $h\in \N \backslash S$ is an {\it isolated gap }of $S$   if
$\{h-1,h+1\}\subseteq S.$ A numerical semigroup  without isolated gaps
is called a {\it perfect numerical semigroup}.

If $F\in \N\backslash \{0\},$   we denote by
$$\scrP(F)=\{S\mid S \mbox{ is a perfect numerical
	semigroup and }\F(S)=F\}.$$
The main aim of this work is to study the set $\scrP(F).$

In order to collect common properties of some families of numerical semigroups, the concept of covariety was introduced in  \cite{covariedades}. A {\it covariety} is  a nonempty family $\CC$ of numerical semigroups that  fulfills  the following conditions:
\begin{enumerate}
	\item[1)]  There exists the minimum of $\CC,$ with respect to set inclusion.
	\item[2)] If $\{S, T\} \subseteq \CC$, then $S \cap  T \in \CC$.
	\item[3)]  If $S \in \CC$  and $S \neq  \min(\CC)$, then $S \backslash \{\m(S)\} \in \CC$.
\end{enumerate}

In this paper, by using the techniques of covarieties, we study the set of $\scrP(F).$ 

The paper is structured as follows. In Section 2, we will see that $\scrP(F)$ is a covariety and its elements can be ordered in a rooted tree. Additionally, we will see how the children of an arbitrary vertice of this tree are. These results will be used in Section 3 to show three algorithms which compute: the set $\scrP(F)$, the maximals elements of $\scrP(F)$ and the elements of $\scrP(F)$ with a given genus. 

We will say that a set $X$ is a $\scrP(F)$-{\it set} if it verifies two conditions:
\begin{enumerate}
	\item $X \cap \min(\scrP(F))=\emptyset.$
	\item There is $S\in \scrP(F)$ such that $X\subseteq S.$
\end{enumerate}

If $X$ is a $\scrP(F)$-set, in Section 4, we prove that   then there exists the  least element of $\scrP(F)$ (with respect to set inclusion) containing $X.$ This element will be denoted by $\scrP(F)[X]$ and we will say that $X$ is a $\scrP(F)$-{\it system of generators}. It will be shown that the minimal $\scrP(F)$-system of generators, in general, are not unique.  Given an element $S\in \scrP(F),$ we define the $\scrP(F)$-$\rank$ of $S$ as 
$$
 \scrP(F)\rank(S)=\min\{\sharp X \mid X \mbox{ is a }\scrP(F)\mbox{-set and }\scrP(F)[X]=S\}. 
$$

We finish  Section 4, characterizing how the elements of $\scrP(F)$ with $\scrP(F)$-$\rank$ $0$, $1$ and $2$ are.

In the semigroup literature one can find a long list of works dedicated to the study of  one dimensional analytically irreducible domains via their value
semigroup.   One of the properties
studied for this kind of rings using this approach has been  to have the Arf property and to be saturated (see \cite{arf}, \cite{lipman}),\cite{zariski}, \cite{pham}, \cite{campillo}, \cite{delgado} and \cite{nunez}. The characterization of Arf rings and saturated rings, via their value semigroup gave rise to the notion of Arf semigroup and saturated numerical semigroup.

%
%
%
%

Following the notation introduced in \cite{perfectos}, a $\Parf$-semigroup (respectively, $\Psat$-semigroup) is a perfect numerical semigroup that in addition is Arf (respectively, saturated). 

Denote by $\Arf(F)=\{S\mid S \mbox{ is an Arf numerical semigroup and } \F(S)=F\},$ $\Sat(F)=\{S\mid S \mbox{ is a saturated numerical semigroup and } \F(S)=F\},$ $\Parf(F)=\{S\mid S \mbox{ is a $\Parf$-numerical semigroup and } \F(S)=F\}$ and  $\Psat(F)=\{S\mid S \mbox{ is a $\Psat$-numerical semigroup and } \F(S)=F\}.$

By \cite{coarf} and \cite{cosat}, we know that $\Arf(F)$ and $\Sat(F)$ are covariety and additionally, $\min(\Arf(F))=\min(\Sat(F))=\min(\scrP(F)).$ This fact will be used in Section 5, to prove that $\Parf(F)$ and $\Psat(F)$ are also covarieties. Moreover, we present some algorithms to compute all the elements of $\Parf(F)$ and $\Psat(F).$

Throughout this paper, some examples are shown to illustrate the results proven. The computation of these examples are performed by using  the GAP (see \cite{GAP}) package \texttt{numericalsgps} (\cite{numericalsgps}).

\section{The covariety $\scrP(F)$ and its associated tree}

Throughout this work, we suppose  $F$ is a positive integer greater than or equal to $2$.   Our first aim will be to prove that $\scrP(F)$ is a covariety. 

The following result has an immediate proof.

\begin{lemma}\label{lemma1}The numerical semigroup 
	$\Delta(F)=\{0,F+1,\rightarrow\},$ where the symbol $\rightarrow$ means that every integer greater than $F+1$ belongs to the set,  is the minimum of $\scrP(F).$
\end{lemma}

The next lemma  is well known and it is very easy to prove. 

\begin{lemma}\label{lemma2}
	Let $S$ and $T$ be numerical semigroups and $x\in S.$ Then the following conditions hold:
	\begin{enumerate}
		\item $S\cap T$ is  a numerical semigroup and $\F(S\cap T)=\max\{\F(S), \F(T)\}.$
		\item $S\backslash \{x\}$ is a numerical semigroup if and only if $x\in \msg(S).$
		\item $\m(S)=\min\left( \msg(S)\right).$
	\end{enumerate}
\end{lemma}

Next we describe a characterization of perfect numerical semigroups that appears in \cite[Proposition 1]{perfectos}.

\begin{lemma}\label{lemma3}
	Let $S$ be a numerical semigroup. The following conditions are
	equivalent.
	\begin{enumerate}
		\item[{\it 1)}] $S$ is a perfect numerical semigroup.
		\item[{\it 2)}] If $\{s,s+2\}\subseteq S$, then $s+1\in S.$
	\end{enumerate}
\end{lemma}

By applying Lemmas \ref{lemma2} and \ref{lemma3}, we can easily deduce the following result. 
\begin{lemma}\label{lemma4}
	If $\{S,T\}\subseteq \scrP(F),$ then $S\cap T \in \scrP(F).$
\end{lemma}
The following lemma  is straightforward
to prove. 
\begin{lemma}\label{lemma5} Let $S\in \scrP(F)$ such that $S\neq \Delta(F).$ Then $S\backslash \{\m(S)\}\in \scrP(F).$	
\end{lemma}
As a consequence of Lemmas \ref{lemma1}, \ref{lemma4} and \ref{lemma5}, we have the following result. 
\begin{proposition}\label{proposition6} Whit the above notation, 
$\scrP(F)$ is a covariety and $\Delta(F)$ is its minimum.	
\end{proposition}

A {\it graph} $G$ is a pair $(V,E)$ where $V$ is a nonempty set and
$E$ is a subset of $\{(u,v)\in V\times V \mid u\neq v\}$. The
elements of $V$ and $E$ are called {\it vertices} and {\it edges},
respectively.

 A {\it path (of
	length $n$)} connecting the vertices $x$ and $y$ of $G$ is a
sequence of different edges of the form $(v_0,v_1),
(v_1,v_2),\ldots,(v_{n-1},v_n)$ such that $v_0=x$ and $v_n=y$.

A graph $G$ is {\it a tree} if there exists a vertex $r$ (known as
{\it the root} of $G$) such that for any other vertex $x$ of $G$
there exists a unique path connecting $x$ and $r$. If  $(u,v)$ is an
edge of the tree $G$, we say that $u$ is a {\it child} of $v$.\\

Define the graph $\G(\scrP(F))$ in the following way:
\begin{itemize}
	\item the set of vertices of $\G(\scrP(F))$ is $\scrP(F)$,
	\item $(S,T)\in \scrP(F) \times \scrP(F)$ is an edge of $\G(\scrP(F))$ if and only if $T=S\backslash \{\m(S)\}.$
\end{itemize}

As a consequence of Propositions \ref{proposition6} and \cite[Proposition 2.6]{covariedades}, we can assert that $\G(\scrP(F))$ is a rooted tree. 
\begin{proposition}\label{proposition7} With the above notation, 
	 $\G(\scrP(F))$ is a tree with root $\Delta(F).$	
\end{proposition}

A tree can be  built recurrently starting from the root  and connecting, 
through an edge, the vertices already built with  their  children. Hence, it is very interesting to characterize the children of an arbitrary vertex of the tree $\G(\scrP(F)).$ For this reason, next  we are going to introduce some  concepts and results that are necessary to understand the work. 

Following the terminology introduced in \cite{JPAA}, an integer $z$ is a {\it pseudo-Frobenius number} of $S$ if $z\notin S$ and $z+s\in S$ for all $s\in S\backslash \{0\}.$  We denote by $\PF(S)$
the set of pseudo-Frobenius numbers of $S.$ The cardinality of $\PF(S)$ is an important invariant of $S$ (see \cite{froberg} and \cite{barucci}) called the {\it type} of $S,$  denoted by $\t(S).$

%

Given a numerical semigroup $S,$ denote by $\SG(S)=\{x\in \PF(S)\mid 2x\in S\}.$
The elements of $\SG(S)$ will be called the {\it special gaps} of $S.$ The following result is Proposition 4.33 from \cite{libro}.

\begin{lemma}\label{proposition9}
	Let $S$ be a numerical semigroup and $x\in \N\backslash S.$ Then $x\in \SG(S)$ if and only if $S\cup \{x\}$ is a numerical semigroup.
\end{lemma}
As a consequence of Proposition  \ref{proposition6} and \cite[Proposition 2.9]{covariedades} , we have the following result. 

\begin{proposition}\label{proposition10} If $S\in \scrP(F),$ then the set formed by the children of $S$ in the tree $\G(\scrP(F))$ is $\{S\cup \{x\}\mid x\in \SG(S), x<\m(S) \mbox { and }S\cup \{x\}\in \scrP(F)\}.$
\end{proposition}

The proof of the following result is immediate.

\begin{lemma}\label{lemma11} Let $S\in \scrP(F),$ $x\in \SG(S)$ and $x<\m(S).$ Then $S\cup \{x\}\in \scrP(F)$ if and only if $x\notin \{2,\m(S)-2,F\}.$	
\end{lemma}

The next result is a consequence from Proposition \ref{proposition10} and Lemma \ref{lemma11}.
\begin{proposition}\label{proposition12} If $S\in \scrP(F),$ then the set formed by the children of $S$ in the tree $\G(\scrP(F))$ is $\{S\cup \{x\}\mid x\in \SG(S), x<\m(S) \mbox { and }x\notin \{2,\m(S)-2,F\}\}.$
\end{proposition}
\section{Three algorithms}

 Our goal in this section is to describe some algorithms which compute:
\begin{enumerate}
	\item The set $ \scrP(F).$
	\item The maximal elements of $ \scrP(F).$
	\item The elements of  $\scrP(F)$ with a fixed genus.
\end{enumerate}

Let $S$ be a numerical semigroup and $n\in S\backslash \{0\}$. The
Apéry set of $n$ in $S$ (named so in honour of \cite{apery}) is defined as
$\Ap(S,n)=\{s\in S\mid s-n \notin S\}$. 

The following result is deduced from \cite[Lemma 2.4]{libro}.

\begin{lemma}\label{lemma10}
	If $S$ is a  numerical semigroup and $n\in S\backslash \{0\},$ Then $\Ap(S,n)$ is a set with cardinality $n.$ Moreover, $\Ap(S,n)=\{0=w(0),w(1), \dots, w(n-1)\}$, where $w(i)$ is the least
	element of $S$ congruent with $i$ modulo $n$, for all $i\in
	\{0,\dots, n-1\}.$
\end{lemma}

Let $S$ be a numerical semigroup. In \cite[Nota 3.5]{covariedades} appears that if we know $\Ap(S,n)$ for some $n\in S\backslash \{0\},$ then we can easily compute $\SG(S).$ And in \cite[Nota 3.8]{covariedades} is showed that if $\Ap(S,n)$ is known for some $n\in S\backslash \{0\},$ then it is very easy to compute $\Ap(S\cup \{x\},n)$ for every $x\in \SG(S).$

\begin{algorithm}\mbox{}\par
	\label{algorithm14}
\end{algorithm}
\noindent\textsc{Input}: An integer $F$ greater than or equal to $2$.
  \par
\noindent\textsc{Output}: $\scrP(F).$

\begin{enumerate}
	\item[(1)] $\scrP(F)=\{\Delta(F)\},$ $B=\{\Delta(F)\}$ and  $\Ap(\Delta(F),F+1)=\{0,F+2,\cdots,2F+1\}.$ 
	\item[(2)] For every $S\in B$ compute $\theta(S)=\{x\in \SG(S)\mid x<\m(S) \mbox{ and } x\notin \{2,\m(S)-2,F\}\}.$
	\item[(3)] If   $\displaystyle\bigcup_{S\in B}\theta(S)=\emptyset,$ then return $\scrP(F).$ 	
	\item[(4)] $C=\displaystyle\bigcup_{S\in B}\{S\cup \{x\}\mid x\in \theta(S)\}.$
	\item[(5)]  $\scrP(F)=\scrP(F)\cup C,$  $B=C$, compute $\Ap(S,F+1)$ for every   $S\in C$  and go to Step $(2).$ 	
\end{enumerate}
In the next example, we show how the previous algorithm works. 

\begin{example}\label{example15} We are going to compute $\scrP(7)$, by using Algorithm \ref{algorithm14}.
	\begin{itemize}
	\item $\scrP(7)=\{\Delta(7)\},$ $B=\{\Delta(7)\}$ and $\Ap(\Delta(7),8)=\{0,9,10,11,12,13,14,15\}.$
	\item $\theta(\Delta(7))=\{4,5\}.$
	\item $C=\{\Delta(7) \cup \{4\},\Delta(7) \cup \{5\}\}.$
	\item $\scrP(7)=\{\Delta(7),  \Delta(7) \cup \{4\},\Delta(7) \cup \{5\}\}$, $B=\{\Delta(7) \cup \{4\},\Delta(7) \cup \{5\}\},$ $\Ap(\Delta(7) \cup \{4\},8)=\{0,4,9,10,11,13,14,15\}$ and $\Ap(\Delta(7) \cup \{5\},8)=\{0,5,9,10,11,12,14,15\}.$
	\item $\theta(\Delta(7) \cup \{4\})=\emptyset,$ $\theta(\Delta(7) \cup \{5\})=\{4\}.$ 
	\item $C=\{\Delta(7) \cup \{4,5\}\}.$
	\item $\scrP(7)=\{\Delta(7), \Delta(7) \cup \{4\},\Delta(7) \cup \{5\}, \Delta(7) \cup \{4,5\}\},$  $B=\{\Delta(7) \cup \{4,5\}\}$ and  $\Ap(\Delta(7) \cup \{4,5\},8)=\{0,4,5,9,10,11,14,15\}.$ 
	\item $\theta(\Delta(7) \cup \{4,5\})=\emptyset.$ 
	\item  Therefore, the Algorithm return $\scrP(7)=\{\Delta(7), \Delta(7) \cup \{4\}, \Delta(7) \cup \{5\},\Delta(7) \cup \{4,5\}\}.$
	
	\end{itemize}
	
\end{example}

Denote by $\Max(\scrP(F))$ the set of maximal elements of $\scrP(F).$ Our next aim will be to present two algorithms which allows to compute $\Max(\scrP(F)).$ For this reason, we will need to introduce some concepts and results.

If $S$ is not a perfect numerical semigroup, then we denote by $\h(S)$ its maximum isolated gap. 
The following result appears in \cite[Proposition 25]{perfectos}.

\begin{lemma}\label{lemma16} If $S$ is not a perfect numerical semigroup, then $S \cup \{\h(S)\}$ is a numerical semigroup.	
\end{lemma}
As a consequence of previous lemma, we have the following result.
\begin{lemma}\label{lemma17} If $S$ is a numerical semigroup with Frobenius number $F$ and $F-1\notin S,$ then there exists $T\in \scrP(F)$ such that $S\subseteq T.$	
\end{lemma}

In the next proposition we present a characterization of maximal elements of $\scrP(F).$
\begin{proposition}\label{proposition18} If $S$ is a numerical semigroup, then the following conditions are equivalent.
	\begin{enumerate}
		\item[1)] $S\in \Max(\scrP(F)).$
		\item[2)] $S$ is maximal in the set $\{T\mid T \mbox{ is a numerical semigroup and }T\cap\{F,F-1\}=\emptyset\}.$
	\end{enumerate}
	
\end{proposition}
\begin{proof}
 {\it 1) implies 2).}	
	We suppose that $S$ is not maximal in the set $\{T\mid T \mbox{ is a numerical semigroup and }T\cap\{F,F-1\}=\emptyset\},$ then there exists $T$ numerical semigroup such that $T\cap\{F,F-1\}=\emptyset$ and $S \subsetneq T.$ It is clear that $\F(T)=F$ and $F-1\notin T.$ Therefore, by applying Lemma \ref{lemma17}, there is $P\in \scrP(F)$ such that $T\subseteq P.$ Thus, $S\subsetneq P$ and consequently, $S\notin \Max(\scrP(F)).$
	
	 {\it 2) implies 1).}	First we show that $S\in \scrP(F). $ Otherwise, by applying Lemma \ref{lemma17}, there exists  $T\in \scrP(F)$ such that $S\subsetneq T.$ It is clear that $T\cap\{F,F-1\}=\emptyset$ and so $S$ is not maximal in the set $\{T\mid T \mbox{ is a numerical semigroup and }T\cap\{F,F-1\}=\emptyset\}.$ 
	 
	 Finally, $S\in \Max(\scrP(F))$ because $\scrP(F)\subseteq \{T\mid T \mbox{ is a numerical semigroup}\\\mbox{and }T\cap\{F,F-1\}=\emptyset\}.$	
\end{proof}
If $C\subseteq \N\backslash \{0\},$ then we denote by $\CL(C)=\{S\mid S \mbox{ is a numerical semigroup }\\ \mbox{and }S\cap C=\emptyset\}.$ Denote by $\Max(\CL(C))$ the set formed by the maximal elements of $\CL(C).$ The Algorithm 1 from \cite{particiones} allows to compute the set $\Max(\CL(C))$ from $C.$ Therefore, by using Proposition \ref{proposition18}, we can assert that we have an algorithm to obtain  $\Max(\scrP(F)).$

The following result is deduced from \cite[Lemma 4.35]{libro}.
\begin{lemma}\label{lemma19}
	Let $S$ and $T$ be numerical semigroups such that $S \subsetneq T.$ Then $\max(T\backslash S)\in \SG(S).$	
\end{lemma}

In the following result we show another characterization of the elements of $\Max(\scrP(F))$ by using the set of special gaps.

\begin{proposition}\label{proposition20}If $S$ is a numerical semigroup, then the following conditions are equivalent.
	\begin{enumerate}
		\item[1)] $S\in \Max(\scrP(F)).$
		\item[2)] $\SG(S)=\{F,F-1\}.$ 
	\end{enumerate}	
\end{proposition}
\begin{proof}
	{\it 1) implies 2).}	If $S\in \scrP(F),$ then, it is clear that, $\{F,F-1\}\subseteq \SG(S).$ If   $\SG(S)\neq \{F,F-1\},$ then there is $x\in \SG(S)$ such that $x\notin  \{F,F-1\}.$ Thus, $S\cup \{x\}$ is a numerical semigroup and $\left(S\cup \{x\} \right)\cap \{F,F-1\}=\emptyset.$ By Proposition \ref{proposition18}, we obtain that $S\notin \Max(\scrP(F)).$
	
		{\it 2) implies 1).} We will see that $S$ is a maximal numerical semigroup under the condition that $S\cap \{F,F-1\}=\emptyset.$ Otherwise, there is $T$ a numerical semigroup such that $T\cap \{F,F-1\}=\emptyset$ and $S\subsetneq T.$ By Lemma \ref{lemma19}, we know that $\max(T\backslash S)\in \SG(S)=\{F,F-1\},$ which is absurd. Then, by Proposition \ref{proposition18}, we conclude that $S\in \Max(\scrP(F)).$
\end{proof}

The Algorithm 3.5 from \cite{atomic}, enable us to compute the set $\{S\mid S \mbox{ is a numeri-}\\ \mbox{cal semigroup and } \SG(S)=\{F,F-1\}\}.$ Hence, by Proposition \ref{proposition20}, we have another algorithm to obtain the set $\Max(\scrP(S)).$

Next, we illustrate how the algorithm works.
\begin{example}\label{example21}
By using  Algorithm 3.5 from \cite{atomic}, (see \cite[Example 3.8]{atomic}), we have that 
$\{S\mid S \mbox{ is a numerical semigroup and }\SG(S)=\{11,10\}\}=\{S_1=\{0,6,7,8,9,12,\rightarrow\}, S_2=\{0,4,8,9,12,\rightarrow\}, S_3=\{0,3,6,9,12,\rightarrow\} \}.$ Consequently, by applying Proposition \ref{proposition20}, we have that $\Max(\scrP(11))=\{S_1,S_2,S_3\}.$
\end{example}
\begin{note}\label{note22}
	In  Example \ref{example21}, observe that $\g(S_1)=7$ and $\g(S_2)=\g(S_3)=8.$ Then we can assert that all the elements of $\Max(\scrP(11))$ have not the same genus, in general.  

\end{note}

We can use the following order GAP, to obtain the previous results:

\begin{verbatim}
gap> S1:=NumericalSemigroup(6,7,8,9);
<Numerical semigroup with 4 generators>
gap> Genus(S1);
7
gap> S2:=NumericalSemigroup(4,9,14,15);
<Numerical semigroup with 4 generators>
gap> Genus(S2);
8
gap> S3:=NumericalSemigroup(3,13,14);
<Numerical semigroup with 3 generators>
gap> Genus(S3);
8
\end{verbatim}

If we denote by $\beta(F)=\min\{\g(S)\mid S\in \Max(\scrP(F))\},$ then we can state that $\beta(11)=7.$

We end this section by giving an algorithm which computes all the elements of $\scrP(F)$ with a given genus.

Let $S$ be a numerical semigroup. Define, recursively the {\it associated sequence} to $S$ in the following way: $S_0=S$ and  $S_{n+1} =		S_n\backslash \{\m(S_n)\}$ for all $n\in \N.$

If $S$ is a numerical semigroup, then we denote by $\rN(S)=\{s\in S\mid s<\F(S)\},$ the set of {\it small elements }of $S.$ Its cardinality will be denoted by $\n(S).$ Note that $\g(S)+\n(S)=\F(S)+1.$

If $S$ is a numerical semigroup and $\{S_n\}_{n\in \N}$ is the associated sequence to $S$, then $\Cad(S)=\{S_0,S_1,\cdots, S_{\n(S)-1}\}$  is called {\it the associated chain }to $S.$ It is clear that $S_{\n(S)-1}=\{0,\F(S)+1,\rightarrow\}.$

Next,  if we apply that $\g(S_{i+1})=\g(S_{i})+1$ for all $i\in \{0,\cdots,\n(S)-2\},$ we easily get the following result.

\begin{proposition}\label{proposition23}Under the standing  notation,
	$$\{\g(S)\mid S\in \scrP(F)\}=\{x\in \N\mid \beta(F)\leq x\leq F\}.
	$$
\end{proposition}
We illustrate the previous proposition with an example.
\begin{example}\label{example24} Following Note \ref{note22},
as $\beta(11)=7$ then, by applying Proposition \ref{proposition23}, we can assert that $\{\g(S)\mid S\in \scrP(11)\}=\{7,8,9,10,11\}.$
\end{example}

We now have all the necessary tools to obtain the previously announced algorithm.

\begin{algorithm}\mbox{}\par
	\label{algorithm25}
\end{algorithm}
\noindent\textsc{Input}: Two positive  integers $F$  and $g.$ 
\par
\noindent\textsc{Output}: $\{S\in \scrP(F)\mid \g(S)=g\}.$

\begin{enumerate}
	\item[(1)]If $g>F,$ then return $\emptyset.$	
	\item[(2)] Compute $\beta(F).$
	\item[(3)]If $g<\beta(F),$ then return $\emptyset.$
	\item[(4)] $H=\{\Delta(F)\},$ $i=F.$
	\item[(5)] If $i=g,$ then return $H.$
	\item[(6)]	For every $S\in H$ compute $\theta(S)=\{x\in \SG(S)\mid x<\m(S) \mbox{ and } x\notin \{2,\m(S)-2,F\}\}.$	
	\item[(7)] $H=\displaystyle\bigcup_{S\in H}\{S\cup \{x\}\mid x\in \theta(S)\},$ $i=i-1$ 
 and go to Step $(5).$ 	
\end{enumerate}

\section{$\scrP(F)$-system of generators}


We will say that a  set $X$ is a  $\scrP(F)$-{\it set} if $X\cap \Delta(F)=\emptyset$ and there is  $S\in \scrP(F)$ such that $X\subseteq S.$

If $X$ is a $\scrP(F)$-set, then we  denote by $\scrP(F)[X]$  the intersection of all elements of $\scrP(F)$ containing $X.$ As $\scrP(F)$ is a finite set, then by applying Proposition \ref{proposition6}, the intersection of elements of $\scrP(F)$ is again an element of $\scrP(F).$ Therefore, we can ennounce the following proposition.
\begin{proposition}\label{proposition26}
	Let $X$ be a $\scrP(F)$-set. Then $\scrP(F)[X]$ is the smallest element of $\scrP(F)$ containing $X.$
\end{proposition}
If $X$ is a $\scrP(F)$-set and $S= \scrP(F)[X],$ then we will say that $X$ is a $\scrP(F)$-{\it system of generators} of $S.$ Moreover, if $S\neq \scrP(F)[Y]$ for all $Y\subsetneq X,$ then $X$ will be called  a {\it minimal} $\scrP(F)$-{\it system of generators} of $S.$

Let $S$ be a numerical semigroup. Then we denote by 

$$ 
\sR(S)=\{x\in \msg(S)\mid x<\F(S) \mbox{ and } \{x-1,x+1\}\not \subseteq S \} \cup 
$$
$$\{x\in \msg(S)\mid  \{x-1,x+1\}\subseteq S, x+1 \in \msg(S) \mbox{ and } x+1<\F(S) \}. 
$$
\begin{proposition}\label{proposition27}
Let $S\in \scrP(F).$ Then $\sR(S)$ is a $\scrP(F)$-set and $\scrP(F)[\sR(S)]=S.$
\end{proposition}
\begin{proof}
	It is clear that $\sR(S)$ is a $\scrP(F)$-set and $\sR(S)\subseteq S.$ Therefore, by using Proposition \ref{proposition26}, we have $\scrP(F)[\sR(S)]\subseteq S.$
	
	Let $T=\scrP(F)[\sR(S)]$ and we suppose that $T\subsetneq S.$ Then, there is $a=\min(S\backslash T).$ Obviously $a\in \msg(S)$ and $a<F.$ We distinguish two cases:
	\begin{enumerate}
		\item If $\{a-1,a+1\} \not \subseteq S,$ then by using Lemma \ref{lemma2}, we deduce that $S\backslash \{a\}\in \scrP(F).$ As $T\subseteq S\backslash \{a\},$ then $\sR(S)\subseteq S\backslash \{a\},$ which is absurd because $a\in \sR(S).$
		\item If $\{a-1,a+1\} \subseteq S,$ then by the minimality of $a$, we have that $a-1\in T.$ As $a\notin T$ and $T\in \scrP(F),$ then $a+1\notin T.$ Thus, $a+1\in \msg(S)$ and $a+1<F.$ Consequently, $S\backslash \{a,a+1\}\in \scrP(F)$ and $T \subseteq S\backslash \{a,a+1\}.$ Then $\sR(S)\subseteq S\backslash \{a,a+1\},$ which is absurd because $a\in \sR(S).$ 
	\end{enumerate}
	
\end{proof}
The following result is straightforward to prove.

\begin{lemma}\label{lemma28}
	If $X$ and $Y$ are $\scrP(F)$-sets such that $X\subseteq Y,$ then  $\scrP(F)[X]\subseteq \scrP(F)[Y].$
\end{lemma}
In the following result we present a characterization of a minimal $\scrP(F)$-system of generators.
\begin{lemma}\label{lemma29}
	Let $X$ be a $\scrP(F)$-set and $S=\scrP(F)[X].$ Then $X$ is a  minimal $\scrP(F)$-system of generators of $S$ if and only if $x\notin \scrP(F)[X\backslash \{x\}]$ for all $x\in X.$
\end{lemma}
\begin{proof}
{\it Necessity.} If $x\in \scrP(F)[X\backslash \{x\}],$ then by Proposition \ref{proposition26} we have that \,\,\, every element of $\scrP(F)$ containing  $X\backslash \{x\}$, then it contains $X.$ Therefore,\,\,\,\, $\scrP(F)[X\backslash \{x\}]=\scrP(F)[X].$

	{\it Sufficiency.} If $X$ is not a minimal $\scrP(F)$-system of generators of $S,$ then there exists $Y\subsetneq X$ such that $\scrP(F)[Y]=S.$ If $x\in X\backslash Y,$ then by applying Lemma \ref{lemma28}, we have that $x\in\scrP(F)[Y] \subseteq \scrP(F)[X\backslash \{x\}].$

\end{proof}

In general, the minimal  $\scrP(F)$-systems of generators are not unique. Moreover, they may not even have the same cardinality as we show in the following example.
\begin{example}
	Let $S=\langle 10,11,12,13,14,15,16\rangle=\{0,10,11,12,13,14,15,16,\\ 
	20, \rightarrow\}.$ It is clear that $S\in \scrP(19).$ It is obvious that if $T\in \scrP(19)$ and \\ $\{10,12,14,16\}\subseteq T,$ then $S\subseteq T.$ Hence, $\scrP(19)[\{10,12,14,16\}]=S.$ Furthemore, it is easy to see that $\scrP(19)[\{10,12,14\}]=\{0,10,11,12,13,14,20,\rightarrow\}, $ $\scrP(19)[\{10,12,16\}]=\{0,10,11,12,16,20,\rightarrow\}, $ $\scrP(19)[\{10,14,16\}]=\{0,10,14,\\
	15,16,20,\rightarrow\} $ and $\scrP(19)[\{12,14,16\}]=\{0,12,13,14,15,16,20,\rightarrow\}. $
	Thus, by applying Lemma \ref{lemma29}, we have that $\{10,12,14,16\}$ is a minimal $\scrP(19)$-system of generators of $S.$
	
	Reasoning in a similar way, the reader will have no difficulty in seeing that   $\{10,11,13,15,16\}$ is also a minimal $\scrP(19)$-system of generators of $S.$
\end{example} 

If $S\in \scrP(F),$ then the $\scrP(F)$-$\rank$ de $S$ is defined as
 $\scrP(F)\rank(S)=\min\{ \sharp X\mid X \mbox{ is a } \scrP(F)\mbox{-set and }\scrP(F)[X]=S \}.$
 By applying \cite[Propositions 6.1, 6.2 and 6.4; and Lemma 6.3]{covariedades}, we obtain the following result. 
 
 \begin{proposition}\label{proposition31} If  $S\in \scrP(F)$ then the following conditions hold.
 	
\begin{enumerate}
	\item $\scrP(F)\rank(S)\leq \e(S).$
	\item $\scrP(F)\rank(S)=0$ if and only if $S=\Delta(F).$
	\item  If $S\neq \Delta(F)$ and  $X$ is a $\scrP(F)$-set such that $\scrP(F)[X]=S,$ then 
 $\m(S)\in X.$	
 \item  $\scrP(F)\rank(S)=1$ if and only if $S=\scrP(F)[\{\m(S)\}].$
\end{enumerate} 
\end{proposition}

For integers $a$ and $b,$ we say that $a$ {\it divides} $b$ if there exists an integer $c$ such that $b=ca,$ and we denote this by $a\mid b.$ Otherwise, $a$ {\it does not divide} $b$, and we denote this by $a\nmid b.$

The following result has an easy proof.
\begin{lemma}\label{lemma32}
If $m$ is a positive integer such that $m<F,$ $m\nmid F$ and $m\nmid (F-1),$ then $\{m\}$ is a $\scrP(F)$-set and 	$\scrP(F)[\{m\}]=\langle m \rangle \cup \{F+1,\rightarrow\}.$
\end{lemma}
\begin{proposition}\label{proposition33}
	If $m$ is a positive integer such that  $m<F,$ $m\nmid F$ and $m\nmid (F-1),$ then $S=\langle m \rangle \cup \{F+1,\rightarrow\}\in \scrP(F)$ and $\scrP(F)\rank(S)=1.$ Moreover, every element of $\scrP(F)$
 with  $\scrP(F)$-$\rank$ equal to $1$ has this form. 	
\end{proposition}
\begin{proof}
	By Lemma \ref{lemma32}, we know that $S\in \scrP(F)$ and by Proposition \ref{proposition31} we know that  $\scrP(F)\rank(S)=1.$ If $T\in \scrP(F) $ with $\scrP(F)\rank(T)=1,$ then, Proposition \ref{proposition31} asserts that $\m(T)<F$ and $T=\scrP(F)[\{\m(T)\}].$ Clearly $\m(T)\nmid F$ and $\m(T)\nmid (F-1).$ Finally, by Lemma \ref{lemma32}, we conclude that  $T=\langle \m(T) \rangle \cup \{F+1,\rightarrow\}.$
\end{proof}

Our next goal will be to characterize the elements of $\scrP(F)$ with $\scrP(F)$-$\rank$ equal to $2$. For this purpose we introduce some concepts and results.

If $S$ is a numerical semigroup, we   recursively define the following sequence of numerical semigroups:
\begin{itemize}
	\item $S_0=S$,
	\item $S_{n+1} =\left\{\begin{array}{ll}
		S_n\cup \{\h(S_n)\}  & \mbox{if $S_n$ is not perfect},\\
		S_n & \mbox{otherwise.}
	\end{array}
	\right.$
\end{itemize}

The number of isolated gaps of $S$ will be denoted by $\ii(S).$ The following result appears in \cite[Proposition 26]{perfectos}.
\begin{proposition}\label{proposition34}
	If $S$ is a numerical semigroup and $\{S_n\}_{n\in \N}$ is the sequence previously defined,  then $S=S_0\subsetneq S_1 \subsetneq \dots \subsetneq S_{\ii(S)}.$ Moreover, $S_{\ii(S)}$ is a perfect numerical semigroup and $\sharp \left(S_{k+1}\setminus S_k\right)=1$ for all $k\in \{0,\dots, \ii(S)-1\}.$
\end{proposition}
The numerical semigroup $S_{\ii(S)}$ is called {\it perfect closure }of $S$ and it will denoted by $\rP(S).$ Note that  $\rP(S)$ is the least perfect numerical semigroup that contains $S.$
\begin{lemma}\label{lemma35}
	Let $S\in \scrP(F)$ and $a\in \msg(S)$ such that $\{a-1,a+1\}\not \subseteq S$ and $a<F.$ If $X$ is a $\scrP(F)$-set and $\scrP(F)[X]=S,$ then $a\in X.$
\end{lemma} 
\begin{proof}
	By Lemma \ref{lemma2}, we deduce that $S\backslash \{a\}\in \scrP(F).$ If $a\notin X,$ then $X\subseteq S\backslash \{a\}.$ Therefore, by applying Proposition \ref{proposition26}, we have that $\scrP(F)[X]\subseteq S\backslash \{a\}.$ Consequently, $S\subseteq S\backslash \{a\},$ which is absurd.
\end{proof}

Now let us define the ratio of a numerical semigroup. This concept  will be  need in the proof of the following proposition.

 Let $S$ be a numerical semigroup such that $S\neq \N,$ the {\it ratio }of $S$ is defined as 
 $\r(S)=\min\{s\in S\mid \m(S)\nmid s \}.$ Note that
 $\r(S)=\min(\msg(S)\backslash \{\m(S)\}).$

\begin{proposition}\label{proposition36}
Let $m$ and $r$ positive integers such that $m<r<F,$ $m\nmid r$ and $\langle m,r \rangle \cap \{F-1,F\}=\emptyset.$ Then $\rP\left(\langle m,r \rangle \cup \{F+1,\rightarrow\} \right)$ is an element of $\scrP(F)$ with $\scrP(F)$-$\rank$ equal to $2.$ Moreover, every element of $\scrP(F)$ with $\scrP(F)$-$\rank$ equal to $2$ has this form.
\end{proposition}
\begin{proof}
	If $T=\langle m,r\rangle \cup \{F+1,\rightarrow\},$ then $T$ is a numerical semigroup with Frobenius number $F$ and $F-1\notin T.$ Thus, $\rP(T)\in \scrP(F).$ As $\rP(T)\neq \Delta(F)$ and $\rP(T)\neq \langle m \rangle \cup \{F+1,\rightarrow\},$ then $\scrP(F)\rank(\rP(T))\ge 2.$  Certainly $\rP(T)$ is the least element of $\scrP(F)$ that contains $\{m,r\}$ Hence, $\rP(T)=\scrP(F)[\{m,r\}]$ and so, $\scrP(F)\rank(\rP(T))\leq 2.$ Consequently, $\scrP(F)\rank(\rP(T))=2.$ 
	
	Let $S\in \scrP(F)$ such that  $\scrP(F)\rank(S)=2.$ Then there is a  $\scrP(F)$-set, $X$, with cardinality $2$ such that  $\scrP(F)[X]=S.$  By Proposition \ref{proposition31}, we know that $\m(S)\in X.$ As $\{\r(S)-1,\r(S)+1\}\not \subseteq S$  since it cannot happen that $\m(S)\mid (\r(S)-1)$ and $\m(S)\mid (\r(S)+1),$ then by Lemma \ref{lemma35} we know that $\r(S)\in X.$ Therefore, $X=\{\m(S),\r(S)\}.$ It is clear that $\m(S)<\r(S)<F,$ $\m(S)\nmid \r(S)$ and $\langle \m(S),\r(S)\rangle \cap \{F,F-1\}=\emptyset.$
	
	Finally, as $S=\scrP(F)[X]=\scrP(F)[\{\m(S),\r(S)\}],$ then $S$ is the least element of $\scrP(F)$ containing $\{\m(S),\r(S)\}.$ 
	
	We conclude that  $S=\rP\left( \langle \m(S),\r(S)\rangle \cup 
	\{F+1,\rightarrow\} \right).$	
\end{proof}

Next we illustrate this proposition with an example.
\begin{example}\label{example37}
	Let $m=8,$ $r=11$ and  $F=26.$ Then $8<11<26,$ $8\nmid 11$ and $\langle 8,11 \rangle \cap \{25,26\}=\emptyset.$ By applying Proposition \ref{proposition36}, we have that $\rP\left( \langle 8,11\rangle \cup
	\{27,\rightarrow\} \right)$ is an element of $\scrP(26)$ with $\scrP(26)\rank$ equal to $2.$
	
	Finally, as $ \langle 8,11\rangle \cup
\{27,\rightarrow\}	=\{0,8,11,16,19,22,24,27,\rightarrow\},$ then $$\rP\left( \langle 8,11\rangle \cup 
\{27,\rightarrow\} \right)=\{0, 8,11,16,19,22,23,24,27,\rightarrow\}=\langle 8,11,23,28,29\rangle.$$	
	
\end{example}
\section{ The Arf or saturated elements in $\scrP(F)$ }


We will say that a numerical semigroup $S$ is an {\it Arf numerical semigroup} if $x+y-z\in S$ for all $\{x,y,z\}\subseteq S$ such that $x\ge y \ge z.$ We will denote by $\Arf(F)=\{S\mid S \mbox{ is an Arf numerical semigroup and }\F(S)=F\}.$

Let $A \subseteq \N$ and $a\in A\backslash \{0\}.$ Denote by $\d_A(a)=\gcd\{x\in A\mid x\leq a\}.$ A numerical semigroup is {\it saturated }if $s+\d_S(s)\in S$ for all $s\in S\backslash \{0\}.$ Denote by $\Sat(F)=\{S\mid S \mbox{ is a saturated  numerical semigroup and }\F(S)=F\}.$

In \cite[Lemma 3.31]{libro} it is shown the relation between the saturated numerical semigroups and  the Arf numerical semigroups. That is the following result.

\begin{proposition}\label{proposition38}
	Every saturated numerical semigroup is an Arf numerical semigroup.
\end{proposition}

By applying \cite[Proposition 2.7]{coarf} and \cite[Proposition 3.3]{cosat}, we obtain the following result. 
\begin{proposition}\label{proposition39}
	Under the standing  notation, we have that $\Arf(F)$  and  $\Sat(F)$  are covarieties and $\Delta(F)$ is their  minimum.
\end{proposition}

Following the notation introduced in \cite{perfectos} a  $\Parf$-semigroup (respectively,  $\Psat$) is a perfect numerical semigroup which in addition it is Arf (respectively, saturated).

Denote by $\Parf(F)=\{S\mid S \mbox{ is a $\Parf$-semigroup and }\F(S)=F\}$ and  $\Psat(F)$\\
$=\{S\mid S \mbox{ is a $\Psat$-semigroup and }\F(S)=F\}.$ Our next aim in this section will be to prove that $\Parf(F)$ and $\Psat(F)$ are covarieties.

In \cite[Lemma 5.1]{covariedades} appears the following result.
\begin{lemma}\label{lemma40}
	Let $\{\CC_i\}_{i\in I}$ be a family of covarieties with $\min(\CC_i)=\Delta$ for  all $i\in I.$ Then $\bigcap_{i\in I}\CC_i$ is a covariety with minimum $\Delta.$
\end{lemma}

By Proposition \ref{proposition6}, \cite[Proposition 2.7]{coarf} and \cite[Proposition 3.3]{cosat}, we know that $\scrP(F),$ $\Arf(F)$ and $\Sat(F)$ are covarieties with minimum $\Delta(F).$ Then by applying Lemma \ref{lemma40}, we have the following result.

\begin{proposition}\label{proposition41}	Under the standing  notation,   $\Parf(F)$ and $\Psat(F)$ are covarieties with minimum $\Delta(F).$
\end{proposition}

Our next purpose will be to present some algorithms to compute the covarieties $\Parf(F)$ and $\Psat(F).$ The following result appears in \cite[Proposition 2.9]{covariedades}.

\begin{lemma}\label{lemma42}
	Let $\CC$ be a covariety and $S\in \CC.$ Then the set formed by the children of $S$ in the tree $\G(\CC)$ is $\{S\cup \{x\}\mid x\in \SG(S), x<\m(S) \mbox{ and }S\cup \{x\}\in \CC \}.$
\end{lemma}

Let $\sA$ and $\sB$ be two covarieties such that $\sB \subseteq \sA$ and $S\in \sB.$  Denote by $
\alpha(\sA,\sB,S)=\min\{x\in\SG(S) \mid x<\m(S) \mbox{ and }S\cup \{x\}\in \sB\}
.$
Depending of the existence of $\alpha(\sA,\sB,S),$ we define

 $$\L(S) =\left\{\begin{array}{ll}
	S\cup \{\alpha(\sA,\sB,S)\}  & \mbox{if there is } \alpha(\sA,\sB,S),\\
	S & \mbox{otherwise.}
\end{array}\right.$$

Define the sequence $\hat{S_0}=S$ and  $\hat{S}_{n+1}=\L(\hat{S}_n)$ for all $n\in \N.$ Obviously,  there exists $\l(\sA,\sB,S)=\min\{x\in \N\mid \L(\hat{S}_x)=\hat{S}_x\}.$

Let $S$ be a numerical semigroup. Then we define the {\it associated sequence } to $S$ in the following form: $S_0=S$ and $S_{n+1}=S_n\backslash \{\m(S_n)\}.$

Let $\sA$ be  a covariety, $S\in \sA$ and $\{S_n\}_{n\in \N}$ the associated sequence to $S.$ Then it is clear that there exists $\C(\sA,S)=\min\{n\in \N\mid S_n=\min(\sA)\}.$ Denote by $\Cad_{\sA}(S)=\{S_0,S_1,\cdots, S_{\C(\sA,S)}\}.$

\begin{proposition}\label{proposition43} 
	Let $\sA$ and $\sB$ be covarieties such that $\sB\subseteq \sA$ and $\min(\sB)=\min(\sA).$ If {\Large $\gamma$}$=\{S\in \sB\mid S \mbox{ has not children in the tree }\G(\sA) \mbox{ which belong to } \sB\},$ then $\sB=\bigcup_{S\in \gamma} \Cad_{\sA}(S).$
	
\end{proposition}
\begin{proof}
	As $\sB$ is a covariety, {\Large $\gamma$}$ \subseteq \sB$ and $\min(\sB)=\min(\sA),$ then we easily deduce that $\bigcup_{S\in \gamma} \Cad_{\sA}(S)\subseteq \sB.$ For the other inclusion,  if $S\in \sB$ and $\{\hat{S}_n\}_{n\in\N}$ is the sequence defined previously, then $\hat{S}_{\l(\sA,\sB,S)}\in$ {\Large $\gamma$} and $S\in \Cad_{\sA}(\hat S_{\l(\sA,\sB,S)}).$
\end{proof}
As an immediate consequence of Proposition \ref{proposition43}, we have the following result.
\begin{corollary}\label{corollary44}Under the previous notation,
	\begin{enumerate}
		\item If \,\,{\Large $\gamma$}$=\{S\in \Parf(F)\mid S \mbox{ has not children in the tree }\G(\scrP(F))\mbox{ which be-}\\ \mbox{long to }\Arf(F)\},$ then $\Parf(F)=\bigcup_{S\in \gamma}\Cad_{\scrP(F)}(S).$
		\item If \,\, {\Large $\gamma$}$=\{S\in \Psat(F)\mid S \mbox{ has not children in the tree }\G(\scrP(F))\mbox{ which be-}\\ \mbox{long to }\Sat(F)\},$ then $\Psat(F)=\bigcup_{S\in \gamma}\Cad_{\scrP(F)}(S).$
		
	\end{enumerate}

\end{corollary}

If $S\in \scrP(F),$ then Algorithm 1 of \cite{coarf}, allows us to determine if a child of $S$ in the tree $\G(\scrP(F))$ is an element of $\Arf(F).$ Therefore, we have an algorithm to compute the set $\Parf(F).$ In a similar way, if $S\in \scrP(F),$  Proposition 3.6 of \cite{cosat}, allows  to determine if a child of $S$ in the tree $\G(\scrP(F))$ is an element of $\Sat(F).$ Therefore, we have an algorithm to compute the set $\Psat(F).$

As a consequence of Proposition \ref{proposition43}, we have the following result.

\begin{corollary}\label{corollary45}Under the asumption notation,
	\begin{enumerate}
		\item If \,\,{\Large $\gamma$}$=\{S\in \Parf(F)\mid S \mbox{ has not children in the tree }\G(\Arf(F))\mbox{ which be-}\\ \mbox{long to }\scrP(F)\},$ then $\Parf(F)=\bigcup_{S\in \gamma}\Cad_{\Arf(F)}(S).$
		\item If \,\,{\Large $\gamma$}$=\{S\in \Psat(F)\mid S \mbox{ has not children in the tree }\G(\Sat(F))\mbox{ which be-}\\ \mbox{long to }\scrP(F)\},$ then $\Psat(F)=\bigcup_{S\in \gamma}\Cad_{\Sat(F)}(S).$
		
	\end{enumerate}

\end{corollary}

We have implemented the gap order \texttt{IsPerfectNumericalSemigroup}, which allows us to know if a  numerical semigroup is perfect. The input is  the minimal system of generators of the numerical semigroup. 

We will see an example to illustrate how this order is used. If we want to know if the numerical semigroups $S_1=\langle 2,3 \rangle$ and $S_2=\langle 4,5,11 \rangle$ are perfect numerical semigroups, we will use the following orders, respectively:
\begin{verbatim}
	gap> IsPerfectNumericalSemigroup([2,3]);
	false
	gap> IsPerfectNumericalSemigroup([4,5,11]);
	true	
\end{verbatim}

Thererfore, by using this order and \cite[Algorithm 2]{coarf} or  \cite[Algorithm 3.10]{cosat}, we get easily an algorithm to compute the set $\Parf(F)$ or $\Psat(F)$ respectively.

\end{document}